\documentclass{amsart}
\usepackage{setspace}
\usepackage{a4}
\usepackage{amsthm}
\usepackage{latexsym}
\usepackage{amsfonts}
\usepackage{graphicx}
\usepackage{textcomp}
\usepackage{cite}
\usepackage{enumerate}
\usepackage{amssymb}
\usepackage{hyperref}
\usepackage{amsmath}
\usepackage{tikz}
\usepackage[mathscr]{euscript}
\usepackage{mathtools}
\newtheorem{theorem}{Theorem}[section]

\newtheorem{corollary}[theorem] {Corollary}
\newtheorem{definition}[theorem]{Definition}

\newtheorem{problem}[theorem]{Problem}

\newtheorem{question}[theorem]{Question}
\title{This is the title}
\usepackage{fancyhdr}

\begin{document}
\hrule\hrule\hrule\hrule\hrule
\vspace{0.3cm}	
\begin{center}
{\bf{FUNCTIONAL DEUTSCH  UNCERTAINTY PRINCIPLE}}\\
\vspace{0.3cm}
\hrule\hrule\hrule\hrule\hrule
\vspace{0.3cm}
\textbf{K. MAHESH KRISHNA}\\
Post Doctoral Fellow \\
Statistics and Mathematics Unit\\
Indian Statistical Institute, Bangalore Centre\\
Karnataka 560 059, India\\
Email: kmaheshak@gmail.com\\

Date: \today
\end{center}

\hrule\hrule
\vspace{0.5cm}
%--------------------------------------
\textbf{Abstract}: Let $\{f_j\}_{j=1}^n$ and $\{g_k\}_{k=1}^m$ be  Parseval p-frames  for a finite dimensional Banach space $\mathcal{X}$. Then   we show that 
\begin{align}\label{UE}
 \log (nm)\geq S_f (x)+S_g (x)\geq -p \log \left(\displaystyle\sup_{y \in \mathcal{X}_f\cap \mathcal{X}_g, \|y\|=1}\left(\max_{1\leq j\leq n, 1\leq k\leq m}|f_j(y)g_k(y)|\right)\right), \quad \forall x \in \mathcal{X}_f\cap \mathcal{X}_g,
\end{align}
where 
\begin{align*}
&\mathcal{X}_f\coloneqq \{z\in \mathcal{X}: f_j(z)\neq 0, 1\leq j \leq n\}, \quad  \mathcal{X}_g\coloneqq \{w\in \mathcal{X}: g_k(w)\neq 0, 1\leq k \leq m\},\\
&S_f (x)\coloneqq -\sum_{j=1}^{n}\left|f_j\left(\frac{x}{\|x\|}\right)\right|^p\log \left|f_j\left(\frac{x}{\|x\|}\right)\right|^p, \quad   S_g (x)\coloneqq -\sum_{k=1}^{m}\left|g_k\left(\frac{x}{\|x\|}\right)\right|^p\log \left|g_k\left(\frac{x}{\|x\|}\right)\right|^p, \quad  \forall x \in \mathcal{X}_g.
\end{align*}
 We call Inequality (\ref{UE}) as \textbf{Functional Deutsch Uncertainty Principle}. For Hilbert spaces, we show that Inequality (\ref{UE})  reduces to the uncertainty principle  obtained by Deutsch \textit{[Phys. Rev. Lett., 1983]}. We also derive a dual of Inequality  (\ref{UE}).

\textbf{Keywords}:   Uncertainty Principle, Orthonormal Basis, Parseval Frame, Hilbert space, Banach space.

\textbf{Mathematics Subject Classification (2020)}: 42C15.\\

\hrule

\tableofcontents
\hrule
\section{Introduction}

Let $d \in \mathbb{N}$ and  $~\widehat{}:\mathcal{L}^2 (\mathbb{R}^d) \to \mathcal{L}^2 (\mathbb{R}^d)$ be the unitary Fourier transform obtained by extending uniquely the bounded linear operator 
\begin{align*}
	\widehat{}:\mathcal{L}^1 (\mathbb{R}^d)\cap  \mathcal{L}^2	 (\mathbb{R}^d) \ni f \mapsto \widehat{f} \in  C_0(\mathbb{R}^d); \quad \widehat{f}: \mathbb{R}^d \ni \xi \mapsto \widehat{f}(\xi)\coloneqq \int\limits_{\mathbb{R}^d}	f(x)e^{-2\pi i  \langle x, \xi \rangle}\,dx\ \in \mathbb{C}.
\end{align*}
The  \textbf{Shannon entropy}  at a function  $f \in \mathcal{L}^2 (\mathbb{R}^d) \setminus \{0\}$ is defined as 

\begin{align*}
	S(f)\coloneqq  -\int\limits_{\mathbb{R}^d}	\left|\frac{f(x)}{\|f\|} \right|^2\log \left|\frac{f(x)}{\|f\|} \right|^2\,dx
\end{align*} 
 (with the convention $0\log0=0$) \cite{SHANNON}. In 1957, Hirschman proved the following result \cite{HIRSCHMAN}. 
\begin{theorem}\cite{HIRSCHMAN} (\textbf{Hirschman Inequality})
For all	$f \in \mathcal{L}^2 (\mathbb{R}^d) \setminus \{0\}$, 
\begin{align}\label{H}
	S(f)+S(\widehat{f})\geq 0.
\end{align}
\end{theorem}
In the same paper \cite{HIRSCHMAN} Hirschman conjectured that Inequality (\ref{H}) can be improved to 
\begin{align}\label{HC}
	S(f)+S(\widehat{f})\geq d(1-\log 2), \quad f \in \mathcal{L}^2 (\mathbb{R}^d) \setminus \{0\}.
\end{align}
Inequality (\ref{HC}) was proved independently in 1975  by Beckner  \cite{BECKNER} and Bialynicki-Birula and Mycielski \cite{BIALYNICKIBIRULA}.
\begin{theorem}\cite{BECKNER, BIALYNICKIBIRULA} (\textbf{Hirschman-Beckner-Bialynicki-Birula-Mycielski Uncertainty Principle})
	For all	$f \in \mathcal{L}^2 (\mathbb{R}^d) \setminus \{0\}$, 
	\begin{align*}
		S(f)+S(\widehat{f})\geq d(1-\log 2).
	\end{align*}
\end{theorem}
Now one naturally asks whether there is a finite dimensional version of Shannon entropy and uncertainty principle. 
Let $\mathcal{H}$ be a finite dimensional Hilbert space. Given an orthonormal basis  $\{\tau_j\}_{j=1}^n$ for $\mathcal{H}$, the \textbf{ (finite) Shannon entropy}  at a point $h \in \mathcal{H}_\tau$ is defined as 
\begin{align*}
	S_\tau (h)\coloneqq - \sum_{j=1}^{n} \left|\left \langle \frac{h}{\|h\|}, \tau_j\right\rangle \right|^2\log \left|\left \langle \frac{h}{\|h\|}, \tau_j\right\rangle \right|^2\geq 0,
\end{align*}
where $\mathcal{H}_\tau\coloneqq \{h \in \mathcal{H}: \langle h , \tau_j \rangle \neq 0, 1\leq j \leq n\}$ \cite{DEUTSCH}. In 1983, Deutsch derived following uncertainty principle for Shannon entropy which is fundamental to several developments in Mathematics and   Physics \cite{DEUTSCH}.
\begin{theorem}\cite{DEUTSCH} (\textbf{Deutsch Uncertainty Principle})  \label{DU}
Let $\{\tau_j\}_{j=1}^n$,  $\{\omega_j\}_{j=1}^n$ be two orthonormal bases for a  finite dimensional Hilbert space $\mathcal{H}$. Then 
	\begin{align}\label{DUP}
2 \log n \geq S_\tau (h)+S_\omega (h)\geq -2 \log \left(\frac{1+\displaystyle \max_{1\leq j, k \leq n}|\langle\tau_j , \omega_k\rangle|}{2}\right)	\geq 0, \quad \forall h \in \mathcal{H}_\tau.
	\end{align}
\end{theorem}
Recently, author derived Banach space versions of  Donoho-Stark-Elad-Bruckstein-Ricaud-Torrésani uncertainty principle \cite{KRISHNA1},  Donoho-Stark approximate support  uncertainty principle \cite{KRISHNA3} and Ghobber-Jaming uncertainty principle \cite{KRISHNA2}. We then naturally ask what is the Banach space version of Inequality (\ref{DUP})?
In this paper, we are going to answer this question.

\section{Functional Deutsch   Uncertainty Principle}
In the paper,   $\mathbb{K}$ denotes $\mathbb{C}$ or $\mathbb{R}$ and $\mathcal{X}$ denotes a  finite dimensional Banach space over $\mathbb{K}$. Dual of $\mathcal{X}$ is denoted by $\mathcal{X}^*$. We need the notion of Parseval p-frames for Banach spaces.
\begin{definition}\cite{ALDROUBISUNTANG, CHRISTENSENSTOEVA}\label{A}
	Let $\mathcal{X}$  be a  finite dimensional Banach space over $\mathbb{K}$.  A collection $\{f_j\}_{j=1}^n$  in  $\mathcal{X}^*$  is said to be a \textbf{Parseval p-frame} ($1\leq p <\infty$) for $\mathcal{X}$ if
			\begin{align}\label{PF}
			\|x\|^p=\sum_{j=1}^n|f_j(x)|^p, \quad \forall x \in \mathcal{X}.
		\end{align}
\end{definition}
Note that (\ref{PF}) says that $\|f_j\|\leq 1$ for all $1\leq j \leq n$. Given a Parseval p-frame $\{f_j\}_{j=1}^n$   for $\mathcal{X}$, we define the \textbf{(finite) p-Shannon entropy} at a point $x \in \mathcal{X}_f$ as 

\begin{align*}
	S_f(x)\coloneqq -\sum_{j=1}^{n}\left|f_j\left(\frac{x}{\|x\|}\right)\right|^p\log \left|f_j\left(\frac{x}{\|x\|}\right)\right|^p\geq 0, 
\end{align*}
where $\mathcal{X}_f\coloneqq \{x \in \mathcal{X}:f_j(x)\neq 0, 1\leq j \leq n\}$.
Following is the fundamental  result of this paper. 
\begin{theorem} \label{FDS}(\textbf{Functional Deutsch Uncertainty Principle}) Let $\{f_j\}_{j=1}^n$ and $\{g_k\}_{k=1}^m$ be  Parseval p-frames  for a finite dimensional Banach space $\mathcal{X}$. Then 
	\begin{align*}
	\frac{1}{(nm)^\frac{1}{p}}	\leq \displaystyle\sup_{y \in \mathcal{X}, \|y\|=1}\left(\max_{1\leq j\leq n, 1\leq k\leq m}|f_j(y)g_k(y)|\right)
	\end{align*}
and 
	\begin{align}\label{FD}
		 \log (nm)\geq S_f (x)+S_g (x)\geq -p \log \left(\displaystyle\sup_{y \in \mathcal{X}_f\cap \mathcal{X}_g, \|y\|=1}\left(\max_{1\leq j\leq n, 1\leq k\leq m}|f_j(y)g_k(y)|\right)\right)> 0, \quad \forall x \in \mathcal{X}_f \cap \mathcal{X}_g.
	\end{align}
\end{theorem}
\begin{proof}
	Let $z \in \mathcal{X}$ be such that $\|z\|=1$. Then 
	\begin{align*}
		1&=\left(\sum_{j=1}^n|f_j(z)|^p\right)\left(\sum_{k=1}^m|g_k(z)|^p\right)=\sum_{j=1}^n\sum_{k=1}^m
	|f_j(z)g_k(z)|^p\\
	&\leq \sum_{j=1}^n\sum_{k=1}^m\left(\displaystyle\sup_{y \in \mathcal{X}, \|y\|=1}\left(\max_{1\leq j\leq n, 1\leq k\leq m}|f_j(y)g_k(y)|\right)\right)^p\\
	&=\left(\displaystyle\sup_{y \in \mathcal{X}, \|y\|=1}\left(\max_{1\leq j\leq n, 1\leq k\leq m}|f_j(y)g_k(y)|\right)\right)^pmn
\end{align*}
which gives 
\begin{align*}
	\frac{1}{mn}\leq \left(\displaystyle\sup_{y \in \mathcal{X}, \|y\|=1}\left(\max_{1\leq j\leq n, 1\leq k\leq m}|f_j(y)g_k(y)|\right)\right)^p.
\end{align*}
Since  $1=\sum_{j=1}^n\left|f_j\left(\frac{x}{\|x\|}\right)\right|^p$ for all  $x \in \mathcal{X}\setminus \{0\}$,  $1=\sum_{k=1}^m\left|g_k\left(\frac{x}{\|x\|}\right)\right|^p$ for all  $x \in \mathcal{X}\setminus \{0\}$ and $\log$ function is concave, using Jensen's inequality (see \cite{STEELE}) we get 
\begin{align*}
	S_f (x)+S_g (x)&=\sum_{j=1}^{n}\left|f_j\left(\frac{x}{\|x\|}\right)\right|^p\log \left(\frac{1}{\left|f_j\left(\frac{x}{\|x\|}\right)\right|^p}\right) +	\sum_{k=1}^{m}\left|g_k\left(\frac{x}{\|x\|}\right)\right|^p\log \left(\frac{1}{\left|g_k\left(\frac{x}{\|x\|}\right)\right|^p}\right)\\
	&\leq \log \left(\sum_{j=1}^{n}\left|f_j\left(\frac{x}{\|x\|}\right)\right|^p \frac{1}{\left|f_j\left(\frac{x}{\|x\|}\right)\right|^p}\right)+\log \left(\sum_{k=1}^{m}\left|g_k\left(\frac{x}{\|x\|}\right)\right|^p \frac{1}{\left|g_k\left(\frac{x}{\|x\|}\right)\right|^p}\right)\\
	&=\log n+\log m=\log (nm), \quad \forall x \in \mathcal{X}_f \cap \mathcal{X}_g.
\end{align*}
Let $x \in \mathcal{X}_f \cap \mathcal{X}_g$. Then 

\begin{align*}
	S_f (x)+S_g (x)&=	-\sum_{j=1}^n\sum_{k=1}^m\left|f_j\left(\frac{x}{\|x\|}\right)\right|^p\left|g_k\left(\frac{x}{\|x\|}\right)\right|^p\left[\log \left|f_j\left(\frac{x}{\|x\|}\right)\right|^p+\log \left|g_k\left(\frac{x}{\|x\|}\right)\right|^p\right]\\
&=-\sum_{j=1}^n\sum_{k=1}^m\left|f_j\left(\frac{x}{\|x\|}\right)\right|^p\left|g_k\left(\frac{x}{\|x\|}\right)\right|^p\log \left|f_j\left(\frac{x}{\|x\|}\right)g_k\left(\frac{x}{\|x\|}\right)\right|^p\\
&=-p\sum_{j=1}^n\sum_{k=1}^m\left|f_j\left(\frac{x}{\|x\|}\right)\right|^p\left|g_k\left(\frac{x}{\|x\|}\right)\right|^p\log \left|f_j\left(\frac{x}{\|x\|}\right)g_k\left(\frac{x}{\|x\|}\right)\right|\\
&\geq -p\sum_{j=1}^n\sum_{k=1}^m\left|f_j\left(\frac{x}{\|x\|}\right)\right|^p\left|g_k\left(\frac{x}{\|x\|}\right)\right|^p\log\left(\displaystyle\sup_{y \in \mathcal{X}_f\cap \mathcal{X}_g, \|y\|=1}\left(\max_{1\leq j\leq n, 1\leq k\leq m}|f_j(y)g_k(y)|\right)\right)\\
&=-p\log\left(\displaystyle\sup_{y \in \mathcal{X}_f\cap \mathcal{X}_g, \|y\|=1}\left(\max_{1\leq j\leq n, 1\leq k\leq m}|f_j(y)g_k(y)|\right)\right)\sum_{j=1}^n\sum_{k=1}^m\left|f_j\left(\frac{x}{\|x\|}\right)\right|^p\left|g_k\left(\frac{x}{\|x\|}\right)\right|^p\\
&=-p\log\left(\displaystyle\sup_{y \in \mathcal{X}_f\cap \mathcal{X}_g, \|y\|=1}\left(\max_{1\leq j\leq n, 1\leq k\leq m}|f_j(y)g_k(y)|\right)\right).
\end{align*}
\end{proof}
\begin{corollary}
	Theorem \ref{DU} follows from Theorem \ref{FDS}.
\end{corollary}
\begin{proof}
	Let $\{\tau_j\}_{j=1}^n$,  $\{\omega_j\}_{j=1}^n$ be two orthonormal bases   for a  finite dimensional Hilbert space $\mathcal{H}$.	Define 
	\begin{align*}
		f_j:\mathcal{H} \ni h \mapsto \langle h, \tau_j \rangle \in \mathbb{K}; \quad g_j:\mathcal{H} \ni h \mapsto \langle h, \omega_j \rangle \in \mathbb{K}, \quad \forall 1\leq j\leq n.
	\end{align*}
Now by using Buzano inequality (see \cite{BUZANO, FFUJIIKUBO}) we get 
\begin{align*}
	\displaystyle\sup_{h \in \mathcal{H}, \|h\|=1}\left(\max_{1\leq j, k\leq n}|f_j(h)g_k(h)|\right)&=\displaystyle\sup_{h \in \mathcal{H}, \|h\|=1}\left(\max_{1\leq j,k \leq n}|\langle h, \tau_j \rangle||\langle h, \omega_k \rangle|\right)\\
	&\leq \displaystyle\sup_{h \in \mathcal{H}, \|h\|=1}\left(\max_{1\leq j,k \leq n}\left(\|h\|^2\frac{\|\tau_j\|\|\omega_k\|+|\langle \tau_j, \omega_k \rangle |}{2}\right)\right)\\
	&=\frac{1+\displaystyle\max_{1\leq j, k \leq n}|\langle \tau_j, \omega_k \rangle |}{2}.
\end{align*}
\end{proof}
Theorem  \ref{FDS}  brings the following question.
\begin{question}
	Given $p$, $m$, $n$ and a Banach space $\mathcal{X}$, for which pairs of Parseval p-frames $\{f_j\}_{j=1}^n$ and $\{g_k\}_{k=1}^m$ for $\mathcal{X}$, we have equality in Inequality (\ref{FD})?
\end{question}
Next we derive a dual inequality of (\ref{FD}). For this we need dual of Definition \ref{A}.
\begin{definition}\cite{TEREKHIN, CASAZZACHRISTENSENSTOEVA, TEREKHIN2}
	Let $\mathcal{X}$  be a  finite dimensional Banach space over $\mathbb{K}$.  A collection $\{\tau_j\}_{j=1}^n$  in  $\mathcal{X}$  is said to be a \textbf{Parseval p-frame} ($1\leq p <\infty$) for $\mathcal{X}^*$ if
\begin{align}\label{PFF}
	\|f\|^p=\sum_{j=1}^n|f(\tau_j)|^p, \quad \forall f \in \mathcal{X}^*.
\end{align}	
\end{definition}
Note that (\ref{PFF}) says that 
\begin{align*}
	\|\tau_j\|=\displaystyle \sup_{f\in  \mathcal{X}^*, \|f\|=1}|f(\tau_j)|\leq \sup_{f\in  \mathcal{X}^*, \|f\|=1}\left(\sum_{j=1}^n|f(\tau_j)|^p\right)^\frac{1}{p}=\sup_{f\in  \mathcal{X}^*, \|f\|=1}\|f\|=1, \quad \forall 1\leq j \leq n.
\end{align*}
Given a Parseval p-frame $\{\tau_j\}_{j=1}^n$   for $\mathcal{X}^*$, we define the  \textbf{(finite) p-Shannon entropy} at a point $f \in \mathcal{X}^* _\tau$ as 
\begin{align*}
	S_\tau(f)\coloneqq -\sum_{j=1}^{n}\left|\frac{f(\tau_j)}{\|f\|}\right|^p\log \left|\frac{f(\tau_j)}{\|f\|}\right|^p\geq 0,
\end{align*}
where $\mathcal{X}^*_\tau\coloneqq \{f \in \mathcal{X}^*:f(\tau_j)\neq 0, 1\leq j \leq n\}$.
We now have the following dual to Theorem \ref{FDS}.
\begin{theorem}\label{DUALFDS}
	(\textbf{Functional Deutsch Uncertainty Principle}) Let $\{\tau_j\}_{j=1}^n$ and $\{\omega_k\}_{k=1}^m$ be  two Parseval p-frames  for the dual $\mathcal{X}^*$ of a finite dimensional Banach space $\mathcal{X}$. Then 
	\begin{align*}
		\frac{1}{(nm)^\frac{1}{p}}	\leq \displaystyle\sup_{g \in \mathcal{X}^*, \|g\|=1}\left(\max_{1\leq j\leq n, 1\leq k\leq m}|g(\tau_j)g(\omega_k)|\right)
	\end{align*}
	and 
	\begin{align}\label{DFDS}
	 \log (nm) \geq	S_\tau (f)+S_\omega (f)\geq -p \log \left(\displaystyle\sup_{g \in \mathcal{X}^*_\tau \cap \mathcal{X}^*_\omega, \|g\|=1}\left(\max_{1\leq j\leq n, 1\leq k\leq m}|g(\tau_j)g(\omega_k)|\right)\right)> 0, \quad \forall f \in \mathcal{X}^*_\tau \cap \mathcal{X}^*_\omega.
	\end{align}
\end{theorem}
\begin{proof}
	Let $h \in \mathcal{X}^*$ be such that $\|h\|=1$. Then 
\begin{align*}
	1&=\left(\sum_{j=1}^n|h(\tau_j)|^p\right)\left(\sum_{k=1}^m|h(\omega_k)|^p\right)=\sum_{j=1}^n\sum_{k=1}^m
	|h(\tau_j)h(\omega_k)|^p\\
	&\leq \sum_{j=1}^n\sum_{k=1}^m\left(\displaystyle\sup_{g \in \mathcal{X}^*, \|g\|=1}\left(\max_{1\leq j\leq n, 1\leq k\leq m}|g(\tau_j)g(\omega_k)|\right)\right)^p\\
	&=\left(\displaystyle\sup_{g \in \mathcal{X}^*, \|g\|=1}\left(\max_{1\leq j\leq n, 1\leq k\leq m}|g(\tau_j)g(\omega_k)|\right)\right)^pmn
\end{align*}	
which gives 
\begin{align*}
	\frac{1}{mn}\leq \left(\displaystyle\sup_{g \in \mathcal{X}^*, \|g\|=1}\left(\max_{1\leq j\leq n, 1\leq k\leq m}|g(\tau_j)g(\omega_k)|\right)\right)^p.
\end{align*}
Since  $1=\sum_{j=1}^n\left|\frac{f(\tau_j)}{\|f\|}\right|^p$ for all  $f \in \mathcal{X}^* \setminus \{0\}$,  $1=\sum_{k=1}^m\left|\frac{f(\omega_k)}{\|f\|}\right|^p$ for all  $f\in  \mathcal{X}^*\setminus \{0\}$ and $\log$ function is concave, using Jensen's inequality  we get 
\begin{align*}
S_\tau (f)+S_\omega (f)&=\sum_{j=1}^{n}\left|\frac{f(\tau_j)}{\|f\|}\right|^p\log \left(\frac{1}{\left|\frac{f(\tau_j)}{\|f\|}\right|^p}\right)+\sum_{k=1}^{m}\left|\frac{f(\omega_k)}{\|f\|}\right|^p\log \left(\frac{1}{\left|\frac{f(\omega_k)}{\|f\|}\right|^p}\right)\\
&\leq \log \left(\sum_{j=1}^{n}\left|\frac{f(\tau_j)}{\|f\|}\right|^p\frac{1}{\left|\frac{f(\tau_j)}{\|f\|}\right|^p}\right)+\log \left(\sum_{k=1}^{m}\left|\frac{f(\omega_k)}{\|f\|}\right|^p\frac{1}{\left|\frac{f(\omega_k)}{\|f\|}\right|^p}\right)\\
&=\log n+\log m=\log (nm), \quad \forall f \in \mathcal{X}^*_\tau \cap \mathcal{X}^*_\omega.
\end{align*}
Let $f \in\mathcal{X}^*_\tau \cap \mathcal{X}^*_\omega$. Then 	

\begin{align*}
		S_\tau (f)+S_\omega (f)&=-\sum_{j=1}^n\sum_{k=1}^m\left|\frac{f(\tau_j)}{\|f\|}\right|^p\left|\frac{f(\omega_k)}{\|f\|}\right|^p\left[\log \left|\frac{f(\tau_j)}{\|f\|}\right|^p+\log \left|\frac{f(\omega_k)}{\|f\|}\right|^p\right]\\
		&=-\sum_{j=1}^n\sum_{k=1}^m\left|\frac{f(\tau_j)}{\|f\|}\right|^p\left|\frac{f(\omega_k)}{\|f\|}\right|^p\log \left|\frac{f(\tau_j)}{\|f\|}\frac{f(\omega_k)}{\|f\|}\right|^p\\
		&=-p\sum_{j=1}^n\sum_{k=1}^m\left|\frac{f(\tau_j)}{\|f\|}\right|^p\left|\frac{f(\omega_k)}{\|f\|}\right|^p\log \left|\frac{f(\tau_j)}{\|f\|}\frac{f(\omega_k)}{\|f\|}\right|\\
		&\geq -p \sum_{j=1}^n\sum_{k=1}^m\left|\frac{f(\tau_j)}{\|f\|}\right|^p\left|\frac{f(\omega_k)}{\|f\|}\right|^p\log \left(\displaystyle\sup_{g \in \mathcal{X}^*, \|g\|=1}\left(\max_{1\leq j\leq n, 1\leq k\leq m}|g(\tau_j)g(\omega_k)|\right)\right)\\
		&=-p\log \left(\displaystyle\sup_{g \in \mathcal{X}^*_\tau \cap \mathcal{X}^*_\omega, \|g\|=1}\left(\max_{1\leq j\leq n, 1\leq k\leq m}|g(\tau_j)g(\omega_k)|\right)\right)\sum_{j=1}^n\sum_{k=1}^m\left|\frac{f(\tau_j)}{\|f\|}\right|^p\left|\frac{f(\omega_k)}{\|f\|}\right|^p\\
		&=-p\log \left(\displaystyle\sup_{g \in \mathcal{X}^*_\tau \cap \mathcal{X}^*_\omega, \|g\|=1}\left(\max_{1\leq j\leq n, 1\leq k\leq m}|g(\tau_j)g(\omega_k)|\right)\right).
\end{align*}
\end{proof}
Theorem  \ref{DUALFDS}  again gives the following question.
\begin{question}
	Given $p$, $m$, $n$ and a Banach space $\mathcal{X}$, for which pairs of Parseval p-frames $\{\tau_j\}_{j=1}^n$ and $\{\omega_k\}_{k=1}^m$ for $\mathcal{X}^*$, we have equality in Inequality (\ref{DFDS})?
\end{question}
Author is aware of the improvement of Theorem \ref{DU} by Maassen and Uffink \cite{MAASSENUFFINK}  (cf. \cite{DEMBOCOVERTHOMAS}) (motivated from a conjecture of Kraus \cite{KRAUS}) but unable to derive Maassen-Uffink uncertainty principle from Theorem \ref{FDS}.

Motivated from R\'{e}nyi entropy, we can easily generalize the notion of p-Shannon entropy and as follows. 
Given a Parseval p-frame $\{f_j\}_{j=1}^n$   for $\mathcal{X}$, we define the \textbf{(finite) p-R\'{e}nyi entropy   of order $\alpha \in (0, \infty),$ $\alpha \neq 1$} at a point $x \in \mathcal{X}_f$  as 
\begin{align*}
	R_{f, \alpha}(x)\coloneqq \frac{1}{1-\alpha}\log\left(\sum_{j=1}^{n}\left|f_j\left(\frac{x}{\|x\|}\right)\right|^{p\alpha}\right). 
\end{align*}
Given a Parseval p-frame $\{\tau_j\}_{j=1}^n$   for $\mathcal{X}^*$, we define the  \textbf{(finite) p-R\'{e}nyi  entropy  of order $\alpha \in (0, \infty),$ $\alpha \neq 1$}   at a point $f \in \mathcal{X}^* _\tau$  as 
\begin{align*}
	R_{\tau, \alpha}(f)\coloneqq \frac{1}{1-\alpha}\log\left(\sum_{j=1}^{n}\left|\frac{f(\tau_j)}{\|f\|}\right|^{p\alpha}\right).
\end{align*}
Using L'H\^{o}pital rule, we have 
\begin{align*}
\lim_{\alpha \to 1}R_{f, \alpha}(\cdot)=S_f(\cdot), \quad 	\lim_{\alpha \to 1}R_{\tau, \alpha}(\cdot)=S_\tau(\cdot).
\end{align*}
Theorem \ref{FDS} and Theorem \ref{DUALFDS} the result in following problems.
\begin{problem}
Given a finite dimensional Banach space $\mathcal{X}$, let $\mathcal{P}(\mathcal{X})$ be the set of all finite Parseval p-frames  for  $\mathcal{X}$.	What is the best function $\Psi:((0,1)\cup (1, \infty))\times \mathcal{P}(\mathcal{X})\times \mathcal{P}(\mathcal{X}) \to (0, \infty)$ satisfying the following: If $\{f_j\}_{j=1}^n$ and $\{g_k\}_{k=1}^m$ are   Parseval p-frames  for  $\mathcal{X}$, then 
\begin{align*}
	R_{f, \alpha} (x)+R_{g, \alpha} (x)\geq \Psi (\alpha, \{f_j\}_{j=1}^n, \{g_k\}_{k=1}^m), \quad \forall x \in \mathcal{X}_f \cap \mathcal{X}_g.
\end{align*}
\end{problem}
\begin{problem}
Given a finite dimensional Banach space $\mathcal{X}$, let $\mathcal{P}(\mathcal{X}^*)$ be the set of all finite Parseval p-frames  for  $\mathcal{X}^*$.	What is the best function $\Psi:((0,1)\cup (1, \infty))\times \mathcal{P}(\mathcal{X}^*)\times \mathcal{P}(\mathcal{X}^*) \to (0, \infty)$ satisfying the following: If $\{\tau_j\}_{j=1}^n$ and $\{\omega_k\}_{k=1}^m$ are   Parseval p-frames  for  $\mathcal{X}^*$, then 
\begin{align*}
	R_{\tau, \alpha} (f)+R_{\omega, \alpha} (f)\geq \Psi (\alpha, \{\tau_j\}_{j=1}^n, \{\omega_k\}_{k=1}^m), \quad \forall f \in \mathcal{X}^*_\tau \cap \mathcal{X}^*_\omega.
\end{align*}	
\end{problem}
Based on  breakthrough result of Berta, Christandl, Colbeck, Renes, and Renner \cite{BERTACHRISTANDLCOLBECKRENESRENNER, BERTACHRISTANDLCOLBECKRENESRENNER2} (which is later generalized by Coles and Piani \cite{COLESPIANI}), we also  set the following problem. 
\begin{problem}
What is the finite dimensional Banach space analogue of Berta-Christandl-Colbeck-Renes-Renner Uncertainty Principle (in the presence of quantum memory?)
\end{problem}

 \bibliographystyle{plain}
 \bibliography{reference.bib}

\begin{thebibliography}{10}

\bibitem{ALDROUBISUNTANG}
Akram Aldroubi, Qiyu Sun, and Wai-Shing Tang.
\newblock {$p$}-frames and shift invariant subspaces of {$L^p$}.
\newblock {\em J. Fourier Anal. Appl.}, 7(1):1--21, 2001.

\bibitem{BECKNER}
William Beckner.
\newblock Inequalities in {F}ourier analysis.
\newblock {\em Ann. of Math. (2)}, 102(1):159--182, 1975.

\bibitem{BERTACHRISTANDLCOLBECKRENESRENNER2}
Mario Berta, Matthias Christandl, Roger Colbeck, Joseph~M. Renes, and Renato
  Renner.
\newblock Supplementary information: The uncertainty principle in the presence
  of quantum memory.
\newblock {\em Nature Phys}, pages 1--12, 2010.

\bibitem{BERTACHRISTANDLCOLBECKRENESRENNER}
Mario Berta, Matthias Christandl, Roger Colbeck, Joseph~M. Renes, and Renato
  Renner.
\newblock The uncertainty principle in the presence of quantum memory.
\newblock {\em Nature Phys}, 6:659--662, 2010.

\bibitem{BIALYNICKIBIRULA}
Iwo Bialynicki-Birula and Jerzy Mycielski.
\newblock Uncertainty relations for information entropy in wave mechanics.
\newblock {\em Comm. Math. Phys.}, 44(2):129--132, 1975.

\bibitem{BUZANO}
Maria~Luisa Buzano.
\newblock Generalizzazione della diseguaglianza di {C}auchy-{S}chwarz.
\newblock {\em Rend. Sem. Mat. Univ. e Politec. Torino}, 31:405--409 (1974),
  1971/73.

\bibitem{CASAZZACHRISTENSENSTOEVA}
Pete Casazza, Ole Christensen, and Diana~T. Stoeva.
\newblock Frame expansions in separable {B}anach spaces.
\newblock {\em J. Math. Anal. Appl.}, 307(2):710--723, 2005.

\bibitem{CHRISTENSENSTOEVA}
Ole Christensen and Diana~T. Stoeva.
\newblock {$p$}-frames in separable {B}anach spaces.
\newblock {\em Adv. Comput. Math.}, 18(2-4):117--126, 2003.

\bibitem{COLESPIANI}
Patrick~J. Coles and Marco Piani.
\newblock Improved entropic uncertainty relations and information exclusion
  relations.
\newblock {\em Phys. Rev. A}, 89(2):022112, 2014.

\bibitem{DEMBOCOVERTHOMAS}
Amir Dembo, Thomas~M. Cover, and Joy~A. Thomas.
\newblock Information-theoretic inequalities.
\newblock {\em IEEE Trans. Inform. Theory}, 37(6):1501--1518, 1991.

\bibitem{DEUTSCH}
David Deutsch.
\newblock Uncertainty in quantum measurements.
\newblock {\em Phys. Rev. Lett.}, 50(9):631--633, 1983.

\bibitem{FFUJIIKUBO}
Masatoshi Fujii and Fumio Kubo.
\newblock Buzano's inequality and bounds for roots of algebraic equations.
\newblock {\em Proc. Amer. Math. Soc.}, 117(2):359--361, 1993.

\bibitem{HIRSCHMAN}
I.~I. Hirschman, Jr.
\newblock A note on entropy.
\newblock {\em Amer. J. Math.}, 79:152--156, 1957.

\bibitem{KRAUS}
K.~Kraus.
\newblock Complementary observables and uncertainty relations.
\newblock {\em Phys. Rev. D (3)}, 35(10):3070--3075, 1987.

\bibitem{KRISHNA3}
K.~Mahesh Krishna.
\newblock Functional {D}onoho-{S}tark approximate-support uncertainty
  principle.
\newblock {\em arXiv:2307.01215v1 [math.FA] 1 July}, 2023.

\bibitem{KRISHNA1}
K.~Mahesh Krishna.
\newblock Functional {D}onoho-{S}tark-{E}lad-{B}ruckstein-{R}icaud-{T}orrésani
  uncertainty principle.
\newblock {\em arXiv: 2304.03324v1 [math.FA] 5 April}, 2023.

\bibitem{KRISHNA2}
K.~Mahesh Krishna.
\newblock Functional {G}hobber-{J}aming uncertainty principle.
\newblock {\em arXiv:2306.01014v1 [math.FA] 1 Jun}, 2023.

\bibitem{MAASSENUFFINK}
Hans Maassen and J.~B.~M. Uffink.
\newblock Generalized entropic uncertainty relations.
\newblock {\em Phys. Rev. Lett.}, 60(12):1103--1106, 1988.

\bibitem{SHANNON}
C.~E. Shannon.
\newblock A mathematical theory of communication.
\newblock {\em Bell System Tech. J.}, 27:379--423, 623--656, 1948.

\bibitem{STEELE}
J.~Michael Steele.
\newblock {\em The {C}auchy-{S}chwarz master class : An introduction to the art
  of mathematical inequalities}.
\newblock AMS/MAA Problem Books Series. Mathematical Association of America,
  Washington, DC; Cambridge University Press, Cambridge, 2004.

\bibitem{TEREKHIN2}
P.~A. Terekhin.
\newblock Representation systems and projections of bases.
\newblock {\em Mat. Zametki}, 75(6):944--947, 2004.

\bibitem{TEREKHIN}
P.~A. Terekhin.
\newblock Frames in a {B}anach space.
\newblock {\em Funktsional. Anal. i Prilozhen.}, 44(3):50--62, 2010.

\end{thebibliography}

\end{document}